\documentclass[11pt]{article}

\usepackage{color}
\usepackage{amsthm}
\usepackage{amsmath}
\usepackage{amssymb}
\usepackage[margin=1in]{geometry}
\usepackage{enumerate}
\usepackage{hyperref}

\usepackage{graphicx}
\usepackage{caption}
\usepackage{subcaption}

\usepackage{mathrsfs}

\title{Positive exponents for compositions of random maps}

\theoremstyle{theorem}
\newtheorem{thm}{Theorem}
\newtheorem{cor}[thm]{Corollary}
\newtheorem{lem}[thm]{Lemma}
\newtheorem{prop}[thm]{Proposition}

\theoremstyle{definition}

\newcommand{\N}{\mathbb{N}}
\renewcommand{\P}{\mathbb{P}}
\newcommand{\R}{\mathbb{R}}

\newcommand{\T}{\mathbb T}

\newcommand{\Len}{\operatorname{Len}}

\renewcommand{\a}{\alpha}
\renewcommand{\b}{\beta}

\newcommand{\e}{\epsilon}

\renewcommand{\tilde}{\widetilde}

\newcommand{\uo}{\underline{\omega}}

\newcommand{\Cc}{\mathcal C}

\newcommand{\Leb}{\operatorname{Leb}}

\title{Lyapunov exponents for random perturbations of\\
some area-preserving maps including the standard map}
\author{Alex Blumenthal\thanks{Courant Institute of Mathematical Sciences, New York University, New York, USA.  Email: alex@cims.nyu.edu.}
\and Jinxin Xue\thanks{Department of Mathematics, University of Chicago, Chicago, Illinois, USA, Email: jxue@math.uchicago.edu. This research was supported by NSF Grant DMS-1500897.} 
\and Lai-Sang Young\thanks{Courant Institute of Mathematical Sciences, New York University, New York, USA.  Email: lsy@cims.nyu.edu. This research was supported in part by NSF Grant DMS-1363161.}
}
\date{}

\begin{document}

\maketitle

\abstract{We consider a large class of 2D area-preserving diffeomorphisms that are
not uniformly hyperbolic but have strong hyperbolicity properties on large regions of
their phase spaces. A prime example is the {\it standard map}. Lower bounds for 
Lyapunov exponents of such systems are very hard to estimate, due to the potential
switching of ``stable" and ``unstable" directions. This paper shows 
that with the addition of (very) small random perturbations, one obtains with
relative ease Lyapunov exponents reflecting the geometry of the deterministic maps.}

\section{Introduction}

A signature of chaotic behavior in dynamical systems is sensitive dependence on
initial conditions. Mathematically, this is captured by the {\it positivity of Lyapunov
exponents}: a differentiable map $F$ of a Riemannian manifold $M$ is said to have 
a positive Lyapunov exponent (LE) at $x \in M$ if $\|dF^n_x\|$ grows exponentially fast 
with $n$. This paper is about volume-preserving diffeomorphisms, and 
we are interested in behaviors that occur on positive Lebesgue measure sets. 
Though the study of chaotic systems occupies a good part of smooth ergodic theory, the hypothesis of positive LE is extremely difficult to verify
when one is handed a concrete map defined by a specific equation -- except
where the map possesses a continuous family of invariant cones. 

An example that has come to symbolize the enormity of the challenge is the {\it standard map}, a mapping 
$\Phi=\Phi_L$ of the $2$-torus given by
$$
\Phi (I, \theta) = (I + L \sin \theta, \ \theta + I + L \sin \theta)
$$
where both coordinates $I, \theta$ are taken modulo $2\pi$ and 
$L \in \R$ is a parameter. For $L \gg 1$, the map $\Phi_L$ has strong expansion 
and contraction, their directions separated by clearly defined invariant cones 
on most of the phase space -- except on two narrow strips near $\theta = \pm \pi/2$
on which vectors are rotated violating cone preservation. 
As the areas of these ``critical regions" tend to zero as $L \to \infty$, one might
expect LE to be positive, but this problem has remained unresolved:
{\it no one has been able to prove, or disprove, the positivity of Lyapunov exponents 
for $\Phi_L$ for any one $L$, however large}, 
 in spite of considerable
effort by leading researchers. The best result known \cite{gorodetski2012stochastic} is that the LE of $\Phi_L$ is positive on sets of Hausdorff dimension 2 (which are very far from having positive Lebesgue measure). The presence of elliptic islands, which has been shown
for a residual set of parameters \cite{duarte1994plenty, duarte2008elliptic}, confirms that the obstructions to proving the positivity
of LE are real.

In this paper, we propose that this problem can be more tractable if one accepts
that dynamical systems are inherently noisy. We show,
for a class of 2D maps $F$ that includes the standard map, that by adding a 
very small, independent random perturbation at each step, the resulting 
maps have a positive LE that correctly reflects the rate 
of expansion of $F$ -- provided that $F$ has sufficiently large expansion
to begin with. More precisely, if $\|dF\| \sim L, \ L \gg 1$, on a large portion of the phase space, then random perturbations of size $O(e^{-L^{2-\varepsilon}})$ are sufficient for
guaranteeing a LE $\sim \log L$. 

Our proofs for these results, which are very short compared to previous works
on establishing nonuniform hyperbolicity for deterministic maps 
(e.g. \cite{jakobson1981absolutely, benedicks1985quadratic, benedicks1991Henon, wang2001strange, wang2006nonuniformly, wang2008toward})
are based on the following idea: We view the random process as a Markov chain 
on the projective bundle of the manifold on which the random maps act, and 
represent LE as an integral. Decomposing this integral into a ``good part" and
a ``bad part", we estimate the first leveraging the strong hyperbolicity of the 
unperturbed map, and obtain a lower bound for the second provided
the stationary measure is not overly concentrated in certain ``bad regions". 
We then use
a large enough random perturbation to make sure that the stationary measure
is sufficiently diffused.

We expect that with more work, this method can be extended both to higher
dimensions and to situations where conditions on the unperturbed map are 
relaxed.

\medskip \noindent
{\bf Relation to existing results.} Closest to the present work are the unpublished results of Carleson and Spencer \cite{spencer1985standard, spencer1985conjectures}, 
who showed for very carefully selected parameters $L \gg 1$ of the standard map that 
LE are positive when the map's derivatives are randomly perturbed. 
For comparison, our first result applies to all $L \gg 1$ with a slightly larger
perturbation than in \cite{spencer1985standard}, and our second result
assumes additionally a finite condition on a finite set; we avoid the rather delicate
parameter selection by perturbing the maps themselves, not just
their derivatives. 

Parameter selections similar to
those in \cite{spencer1985standard} were used -- without random perturbations -- 
to prove the positivity of LE 
for the H\'enon maps \cite{benedicks1991Henon}, quasi-periodic cocycles \cite{young1997lyapunov}, and rank-one attractors \cite{wang2008toward}, building on earlier 
techniques in 1D, see e.g. \cite{jakobson1981absolutely, rychlik1988another, benedicks1985quadratic, wang2006nonuniformly}. 
See also \cite{shamis2015bounds}, which estimates LE from below
for Schr\"odinger cocycles over the standard map. 
Relying on random perturbation alone -- without parameter deletion -- are \cite{lian2012positive}, which contains results analogous to ours in 1D, and
\cite{ledrappier2003random}, which applied random rotations to twist maps.
We mention also \cite{berger2012non}, which uses hyperbolic toral automorphisms
in lieu of random perturbations.

Farther from our setting, the literature on LE is vast. Instead of
endeavoring to give reasonable citation of individual papers, let us mention
several {\it categories} of results in the literature that have attracted much attention,
together with a small sample of results in each.  Furstenberg's work \cite{furstenberg1963noncommuting} in the early 60's initiated extensive research on criteria for the LE of random matrix products to be distinct (see e.g. \cite{gol1989lyapunov, guivarc1986products,virtser1980products}).
Similar ideas were exploited to study LE of cocycles over hyperbolic
and partially hyperbolic systems (see e.g. \cite{bonatti2004lyapunov, bonatti2003genericite}), with a generalization to deterministic maps \cite{avila2010extremal}. Unlike the results in the first two paragraphs, these results 
do not give quantitative estimates; they assert only that LE are simple, or nonzero. 

We mention as well the formula of Herman \cite{herman1983methode, knill1992positive} and the related work \cite{avila2002formula}, which use subharmonicity to estimate Lyapunov exponents, and the substantial body of work on 1D Schr\"odinger operators (e.g. { \cite{kotani1984, bourgain2013lyapounov, puig2004cantor, avila2009ten}}). {We also note the $C^1$ genericity of zero Lyapunov exponents of volume-preserving surface diffeomorphisms away from Anosov \cite{bochi2002genericity} and its higher-dimensional analogue \cite{ bochi2005lyapunov}.} Finally, we acknowledge results on the continuity or stability of LE, as in, e.g., \cite{ruelle1979analyticity, hennion1984loi, kifer1982perturbations, bocker2010continuity,ledrappier1991stability}.

\medskip
This paper is organized as follows: We first state and prove two results in a relatively
simple setting: Theorem \ref{thm:lyapEst}, which contains the core idea of this paper,  is proved in
Sections 3 and 4, while Theorem \ref{thm:lyapEst2}, which
shows how perturbation size can be decreased if some mild 
conditions are assumed, is proved in Section 5. 
We also describe a slightly more general setting
which includes the standard map, and observe in Section 6 that the proofs
given earlier in fact apply, exactly as written, to this broader setting.

\section{Results and remarks}

\subsection{Statement of results}

We let $\psi : \mathbb S^1 \to \R$ be a $C^3$ function for which the following hold:
\begin{itemize}
\item[(H1)] $C_{\psi}' = \{\hat x \in \mathbb S^1 : \psi'(\hat x) = 0\}$ and $C_{\psi}'' = \{\hat z \in \mathbb S^1 : \psi''(\hat z) = 0\}$ have finite cardinality.
\item[(H2)] $\min_{\hat x \in C_{\psi}'} |\psi''(\hat x)| > 0$ and $\min_{\hat z \in C_{\psi}''} |\psi'''(\hat z)| > 0$.
\end{itemize}
For $L > 1$ and $ a \in [0,1)$, we define 
$$f=f_{L,a}: \mathbb S^1 \to \R \qquad \mbox{ by } \qquad
f(x) = L \psi(x) + a\ .
$$
Let $\mathbb T^2 = \mathbb S^1 \times \mathbb S^1$ be the $2$-torus. 
The deterministic map to be perturbed is 
\begin{equation} \label{F}
F=F_{L,a}: \mathbb T^2 \to \mathbb T^2 \qquad \mbox{where} \qquad 
F(x,y) = \bigg( \begin{array}{c} f(x) - y \, \ (\text{mod }1) \\ x \end{array} \bigg) \ .
\end{equation}
We have abused notation slightly in Eq (\ref{F}): 
We have made sense of $f(x)-y$
by viewing $y \in \mathbb S^1$ as belonging in $[0,1)$, and have
written ``$z$ (mod $1$)" instead of
$\pi(z)$ where $\pi:\mathbb R \to \mathbb S^1 \cong \mathbb R/\mathbb Z$ is the usual
projection.
Observe that $F$ is an area-preserving diffeomorphism of $\mathbb T^2$.

We consider compositions of random maps
\[
F^n_{\uo} = F_{\omega_n} \circ \cdots \circ F_{\omega_1} \qquad \mbox{for} 
\quad n=1,2, \dots,
\]
where 
$$
F_\omega = F \circ S_\omega \ , \qquad 
S_\omega(x,y)  = ( x + \omega \, (\text{mod }1), y)\ ,
$$
and the sequence $\uo=(\omega_1, \omega_2, \dots)$ is chosen {\it i.i.d.} with respect 
to the uniform distribution $\nu^\e$ on $[- \e, \e]$ for some $\e > 0$. Thus our sample space can be 
written as $\Omega = [- \e, \e]^\N$,
equipped with the probability $\P =\big(\nu^\e\big)^{ \N}$.

Throughout, we let $\text{Leb}$ denote Lebesgue measure on $\T^2$. 

\smallskip
\begin{thm}\label{thm:lyapEst}
Assume $\psi$ obeys (H1),(H2), and fix $a \in [0,1)$.  Then 
\begin{itemize}
\item[(a)] for every $L>0$ and $\e>0$, 
\begin{equation}\label{eq:lyapDefn}
\lambda_1^\e = \lim_{n \to \infty} \frac 1 n \log \| (d F_{\uo}^n)_{(x,y)} \|
\end{equation}
exists and is independent of $(x, y, \uo)$ for every $(x,y) \in \mathbb T^2$ and $\P$-a.e. $\uo \in \Omega$;
\item[(b)] given $\a, \b \in (0,1)$, there is a constant $C = C_{\alpha, \beta} > 0$ such 
that for all $L, \e$ where $L$ is sufficiently large (depending on $\psi, \a, \b)$
and  $\e \geq L^{- CL^{1 - \beta}}$, we have
\[
\lambda_1^\e \geq \a \log L \, .
\]
\end{itemize}
\end{thm}

\smallskip

Theorem \ref{thm:lyapEst}
assumes no information whatsoever on dynamical properties of $F$ beyond its
definition in Eq (\ref{F}).
Our next result shows, under some minimal, easily checkable, condition on
the first iterates of $F$, that the bound above on $\lambda^\e_1$
continues to hold for a significantly smaller $\e$. Let $\mathcal N_c(C_\psi')$ denote the $c$-neighborhood of $C_\psi'$ in $\mathbb S^1$. We formulate the following condition
on $f=f_{L,a}$:

\bigskip \noindent
(H3)$(c)$ \ For any $\hat x, \hat x' \in C_\psi'$, we have that $f \hat x - \hat x' 
(\text{mod }1) \not \in \mathcal N_{c}(C_\psi')$\ .

\medskip 
Observe that for $L$ large, the set of $a$ for which (H3)$(c)$ is satisfied
tends to $1$ as $c \to 0$.

\smallskip
\begin{thm}\label{thm:lyapEst2} Let $\psi$ be as above, and fix an arbitrary 
$c_0>0$. Then given $\a, \b \in (0,1)$, there is a constant $C = C_{\alpha, \beta} > 0$ such that for all $L, a, \e$ where

\smallskip
-- $L$ is sufficiently large (depending on $\psi,c_0,\a, \b$), 

-- $a \in [0,1)$ is chosen so that $f = f_{L, a}$ satisfies {\rm (H3)}$(c_0)$, and

-- $\e \geq  L^{- CL^{2 - \beta}}$,

\smallskip \noindent
then we have
\[
\lambda_1^\e \geq \a \log L \, .
\]
\end{thm}

\medskip \noindent
{\bf \large A slight extension}

\medskip
Let $\psi : \mathbb S^1 \to \R$ be as above. For $L > 0$ and $a \in [0,1)$, 
we write $f_0 = f_{\psi, L, a} = L \psi + a$, and for $\varepsilon > 0$ define
$$
\mathcal U_{\varepsilon, L}(f_0) = \{ f : \mathbb S^1 \to \R \text{ such that }  \| f-f_0\|_{C^3} < L \varepsilon\}\ .
$$ 

We let $C'_f$ and $C''_f$ denote the zeros of $f'$ and $f''$. Below, (H3)$(c)$ is to be
read with $C'_f$ in the place of $C'_\psi$.
We write $F_f(x,y) = (f(x) -y \mbox{ (mod 1) }, x)$ for $f \in \mathcal U_{\varepsilon, L}(f_0)$.

\smallskip

\begin{thm}\label{thm:lyap3} Let $\psi : \mathbb S^1 \to \R$ satisfy (H1), (H2) as before. For $a \in [0,1)$ and $L >1$, let $f_0$ be as defined above. Then there exists $\varepsilon>0$ sufficiently
small so that 

(1) Theorems \ref{thm:lyapEst} and \ref{thm:lyapEst2} hold for $F = F_f$
for all $L > 0$ sufficiently large and $ f \in \mathcal U_{\varepsilon, L}(f_0)$; 

(2) $L$ depends only on $\psi$ as before but $a$ in Theorem \ref{thm:lyapEst2} 
depends on $f$.
\end{thm} 

\smallskip
The Chirikov standard map is defined as follows: a parameter $L > 0$ is fixed, and the map $(I, \theta) \mapsto (\bar I, \bar \theta)$, sending $[0,2 \pi)^2$ into itself, is defined by
\begin{align*}
\bar I &= I + 2 \pi L \sin \theta \, , \\
\bar \theta &= \theta + \bar I = \theta + I + 2 \pi L \sin \theta \, ,
\end{align*}
where both coordinates $I, \theta$ are taken modulo $2 \pi$. 

\begin{cor}[\bf The standard map] \label{cor:standard} Let $L$ be sufficiently large. Then:
\begin{itemize}
\item Theorem \ref{thm:lyapEst} holds for the standard map.
\item If additionally the map $f(x) = L \sin(2 \pi x) + 2 x$ satisfies (H3)(c) for some $c > 0$, then Theorem \ref{thm:lyapEst2} holds for the standard map for this value of $L$.
\end{itemize}
\end{cor}

\smallskip
Theorem \ref{thm:lyap3} and Corollary \ref{cor:standard} are proved in Section 6. {\it All discussions prior to Section 6
pertain to the setting described at the beginning of this section.}


\subsection{Remarks}

\noindent {\bf Remark 1.}  {\it Uniform hyperbolicity on large but non-invariant regions 
of the phase space.} An important property of the deterministic map $F$ is that 
cone fields can be defined on all of $\mathbb T^2$ in such a way
that they are preserved by $dF_{(x,y)}$ for $(x,y)$ in a large but non-invariant 
region in $\T^2$. For example, let 
$C_{\frac15}=\{v=(v_x, v_y): |v_y/v_x| \le \frac15\}$.
Then for $(x,y) \not \in \{|f'| < 10\}$, which by (H1) and (H2) is comprised 
of a finite number of very narrow vertical strips in $\mathbb T^2$ for $L$ large, 
one checks easily that
$dF_{(x,y)}$ maps $C_{\frac15}$ into $C_{\frac15}$, and expands vectors in these cones
uniformly. It is just as easy to see that this cone invariance property 
cannot be extended across the strips in $\{|f'| < 10\}$, and that $F$ is not uniformly hyperbolic. 

These ``bad regions" where the invariant cone property fails
shrink in size as $L$ increases. More precisely, 
let $K_1 > 1$ be such that $|\psi'(x)| \geq K_1^{-1} d(x, C_\psi')$;
that such a $K_1$ exists follows from (H1), (H2) in Sect. 2.1. It is easy to check that
for any $\eta \in (0,1)$,
$$
d(x, C_{\psi}') \geq \frac{K_1}{ L^{1 - \eta}}\  \quad \implies \quad
|f'(x)| \geq L^\eta\ , 
$$ 
and this strong expansion in the $x$-direction 
is reflected in $dF_{(x,y)}$ for any $y$.

We must stress, however, that regardless of how small these ``bad regions" are, the positivity of Lyapunov exponents is not guaranteed for the deterministic map $F$
 -- except for the Lebesgue measure zero set of orbits that never venture into 
 these regions. In general, tangent vectors that have expanded in the good regions
can be rotated into contracting directions when the orbit visits a bad region. 
This is how elliptic islands are formed.

\bigskip \noindent
{\bf Remark 2.}  {\it Interpretation of condition (H3).} We have seen that visiting neighborhoods of 
$V_{\hat x} := \{x=\hat x\}$ for $\hat x \in C'_\psi$ can lead to
a loss in hyperbolicity, yet
at the same time it is unavoidable that the ``typical" orbit will visit these ``bad regions".
Intuitively, it is logical to expect the situation to improve if we do not permit orbits to visit these bad regions two iterates in a row -- except
that such a condition is impossible to arrange: since $F(V_{\hat x'})= \{y=\hat x'\}$, 
it follows that $F(V_{\hat x'})$ meets $V_{\hat x}$ for every $\hat x, \hat x' \in C'_\psi$. 
In Theorem \ref{thm:lyapEst2}, 
we assert that in the case of random maps, to reduce the size of $\e$ 
it suffices to impose the condition that no orbit can be in $C'_\psi \times \mathbb S^1$
for {\it three consecutive iterates}. That is to say, suppose
$F(x_i,y_i)=(x_{i+1},y_{i+1}), \ i=1,2,\dots$. If $x_i, x_{i+1} \in C'_\psi$, 
then $x_{i+2}$ must stay away from $C'_\psi$. This is a rephrasing of (H3). 
Such a condition is both realizable and checkable, as it involves
only a finite number of iterates for a finite set of points.

\bigskip \noindent
{\bf Remark 3.}  {\it Potential improvements.} 
Condition (H3) suggests that one may be able to shrink $\e$ further by imposing
similar conditions on one or two more iterates
of $F$. Such conditions will cause the combinatorics in Section 5 to be more involved, 
and since our $\e$, which is $\sim L^{-L^{2-\b}}$, is already extremely small for large $L$,
we will not pursue these possibilities here.

\section{Preliminaries}\label{Sb=1w}

The results of this section apply to all $L, \e >0$ unless otherwise stated.

\subsection{Relevant Markov chains}

Our random maps system $\{F^n_{\uo}\}_{n \geq 1}$ can be seen as a 
time-homogeneous Markov chain $\mathcal X := \{(x_n, y_n)\}$ given by 
\[
(x_n, y_n) = F^n_{\uo}(x_0, y_0) = F_{\omega_n}(x_{n - 1}, y_{n - 1}) \, .
\]
That is to say, for fixed $\e$, the transition probability starting from $(x,y) \in \mathbb T^2$ is
\[
P((x,y), A) = P^\e((x,y), A) = \nu^\e \{ \omega \in [- \e, \e] : F_\omega(x,y) \in A \} 
\]
for Borel $A \subset \mathbb T^2$. We write $P^{(k)}((x,y), \cdot)$ (or $P^{(k)}_{(x,y)}$) for the corresponding $k$-step transition probability. It is easy to see that for this chain, Lebesgue measure is \emph{stationary},
meaning for any Borel set $A \subset \mathbb T^2$,
\[
\mbox{Leb}(A) = \int P((x,y), A) \, d \mbox{Leb}(x,y) \, . 
\]

Ergodicity of this chain is easy and we dispose of it quickly.

\begin{lem}\label{lem:ergodicityCons}
Lebesgue measure is ergodic.
\end{lem}
\begin{proof} For any $(x,y) \in \mathbb T^2$ and $\omega_1, \omega_2 \in [-\e,\e]$, 
\begin{equation}\label{eq:twoStepAC}
F_{\omega_2} \circ F_{\omega_1}(x,y) = F \circ F \circ S'_{\omega_1, - \omega_2}(x,y) 
\end{equation}
where $S'_{\omega, \omega'}(x,y) = \big( x + \omega \, (\text{mod } 1), y + \omega' (\text{mod } 1) \big)$. That is to say, $P^{(2)}_{(x,y)}$ is supported on
the set $F^2([x-\e, x+\e] \times [y-\e,y+\e])$, on which it is equivalent to
Lebesgue measure. From this one deduces immediately 
that (i) every ergodic stationary measure of $\mathcal X = \{(x_n, y_n)\}$ has a density, 
and (ii) all nearby points in $\mathbb T^2$
are in the same ergodic component. Thus there can be at most one ergodic 
component.
\end{proof}

Part (a) of Theorem \ref{thm:lyapEst} follows immediately from Lemma \ref{lem:ergodicityCons} together with the
Multiplicative Ergodic Theorem for random maps.

\medskip
Next we introduce a Markov chain $\hat {\mathcal X}$ on $\P \T^2$, the projective bundle over $\T^2$. Associating 
$\theta \in \P^1 \cong [0, \pi)$ with the unit vector $u_\theta = (\cos \theta, \sin \theta)$, $F_\omega$ induces a mapping
$\hat F_\omega : \P \T^2 \to \P \T^2$ defined by 
$$
\hat F_\omega (x,y,\theta) = (F_\omega(x,y), 
\theta') \qquad \mbox{where} \qquad u_{\theta'} = 
\pm \frac{(dF_\omega)_{(x,y)} u_\theta}{\|(dF_\omega)_{(x,y)} u_\theta \|}\ .
$$
Here $\pm$ is chosen to ensure that $\theta' \in [0,\pi)$. The Markov chain 
$\hat {\mathcal X} : = \{(x_n,y_n,\theta_n)\}$ is then defined by
$$
(x_n, y_n,\theta_n) = \hat F_{\omega_n}(x_{n-1}, y_{n-1}, \theta_{n-1})\ .
$$
We write $\hat P$ for its transition operator, $\hat P^{(n)}$ for the $n$-step transition transition operator, and use Leb to denote also Lebesgue measure on $\P \T^2$.

For any stationary probability measure $\hat \mu$ of the Markov chain 
$(x_n,y_n,\theta_n)$, define
$$\lambda(\hat \mu) = \int \log \|(dF_\omega)_{(x,y)} u_\theta \| \ d\hat \mu (x,y,\theta) 
\ d\nu^\e (\omega)\ .
$$

\begin{lem} \label{lambda_mu} For any stationary probability measure $\hat \mu$ of
the Markov chain $\hat {\mathcal X}$, we have
$$
\lambda^\e_1 \ge \lambda(\hat \mu)\ .
$$
\end{lem}

\begin{proof} By the additivity of the 
cocycle $(x,y,\theta) \mapsto \log \|d(F_\omega)_{(x,y)} u_\theta\|$, we have, for any $n \in \N$, 
\begin{eqnarray*}
\lambda (\hat \mu) & = & \int \frac1n \log \|(dF_{\uo}^n)_{(x,y)} u_\theta \| \ d\hat \mu (x,y,\theta) \ d (\nu^\e)^n(\uo)\\
& \le & \int \frac1n \log \|(dF_{\uo}^n)_{(x,y)} \| \ d \mbox{Leb} (x,y) \ d (\nu^\e)^n(\uo)\ .
\end{eqnarray*}
That $\hat \mu$ projects to Lebesgue measure on $\T^2$ is used in passing from 
the first to the second line, and the latter converges to 
$\lambda^\e_1$ as $n \to \infty$ by the Multiplicative Ergodic Theorem.
\end{proof}

Thus to prove part (b) of Theorem \ref{thm:lyapEst}, it suffices to prove that 
$\lambda(\hat \mu) \ge \a \log L$ for some $\hat \mu$. Uniqueness of $\hat \mu$
is not required. On the other hand, once we have
shown that $\lambda^\e_1>0$, it will follow that there can be at most
one $\hat \mu$ with $\lambda(\hat \mu)>0$. 
Details are left to the reader.

We remark also that while Theorems \ref{thm:lyapEst}--\ref{thm:lyap3} hold for arbitrarily large values 
of $\e$, we will treat only the case
$\e \leq \frac12$, leaving the very minor
modifications needed for the $\e > \frac12$ case to the reader. 

Finally, we will omit from time to time the notation ``(mod 1)" when the meaning 
is obvious, e.g. instead of the technically correct
but cumbersome $f(x+\omega \mbox{ (mod 1)}) - y \mbox{ (mod 1)}$,
we will write $f(x+\omega) - y$.


\subsection{A 3-step transition}

In anticipation for later use, we compute here the transition probabilities 
$\hat P^{(3)}((x,y,\theta), \cdot)$, also denoted $\hat P^{(3)}_{(x,y,\theta)}$.
Let $(x_0, y_0, \theta_0) \in \P\T^2$ be fixed.
We define 
$$
H=H^{(3)}_{(x_0, y_0, \theta_0)} : [-\e,\e]^3 \to \P\T^2
$$
by
$$
H(\omega_1, \omega_2, \omega_3) = \hat F_{\omega_3} \circ \hat F_{\omega_2}
\circ \hat F_{\omega_1} (x_0, y_0, \theta_0)\ .
$$
Then $\hat P^{(3)}_{(x_0,y_0,\theta_0)} = H_*((\nu^\e)^3)$, the pushforward of
$(\nu^\e)^3$ on $[-\e,\e]^3$ by $H$.  Write $(x_i, y_i,\theta_i) = \hat F_{\omega_i} (x_{i-1}, y_{i-1},\theta_{i-1})$, $i = 1,2,3$. 

\smallskip
\begin{lem} \label{comp} Let $\e \in (0, \frac12]$. Let $(x_0, y_0, \theta_0) \in \P\T^2$ be fixed, and let
$H= H^{(3)}_{(x_0, y_0, \theta_0)}$ be as above. Then 
\begin{itemize}
\item[(i)] \begin{equation}\label{eq:detdG}
\det dH(\omega_1, \omega_2, \omega_3) =  \sin^2 (\theta_3) \tan^2(\theta_2) \tan^2(\theta_1) f''(x_0 + \omega_1)  \ ;
\end{equation}
\item[(ii)] assuming $\theta_0 \ne \pi/2$, we have that
$\det dH \ne 0$ on $V$ where $V \subset
[-\e,\e]^3$ is an open and dense set having full Lebesgue measure in $[-\e,\e]^3$;
\item[(iii)] $H$ is at most $\#(C_\psi'')$-to-one, i.e., no point in $\P \T^2$ has more than $\# C_\psi''$ preimages.
\end{itemize}
\end{lem}

\begin{proof}[Proof of Lemma \ref{comp}.] The projectivized map $\hat F_\omega$ can be written as
\begin{equation} \label{Fhat}
\hat F_\omega(x,y,\theta)=\left(f(x+\omega)-y\ ,x+\omega,\ \arctan \frac{1}{f'(x+\omega)-\tan\theta} \right)\,  ,
\end{equation}
where $\arctan$ is chosen to take values in $[0,\pi)$.

\noindent
(i) It is convenient to write $k_i=\tan\theta_i$, so that $k_{i + 1} = (f'(y_{i + 1}) - k_i)^{-1}$. Note as well that $x_{i + 1} = f(y_{i + 1}) - y_i$. Then
\begin{equation}
\begin{aligned}
dx_3\wedge dy_3\wedge d\theta_3 
= & (f'(y_3)dy_3-dy_2)\wedge dy_3\wedge \left(\frac{\partial \theta_3}{\partial y_3}dy_3+\frac{\partial \theta_3}{\partial k_2}dk_2
\right)\\
=&-dy_2\wedge dy_3\wedge\left(\frac{\partial \theta_3}{\partial k_2}dk_2\right)\\
= & -dy_2\wedge(d\omega_3+f'(y_2)dy_2-dy_1)\wedge \left(\frac{\partial \theta_3}{\partial k_2}\right)\left(\frac{\partial k_2}{\partial y_2}dy_2+\frac{\partial k_2}{\partial k_1}dk_1\right)\\
=& -dy_2\wedge d(\omega_3-y_1)\wedge \left(\frac{\partial \theta_3}{\partial k_2}\frac{\partial k_2}{\partial k_1}dk_1\right)\\
=& -(d\omega_2+f'(y_1)dy_1)\wedge d(\omega_3-y_1)\wedge \left(\frac{\partial \theta_3}{\partial k_2}\frac{\partial k_2}{\partial k_1}\frac{\partial k_1}{\partial y_1} dy_1\right)\\
=&-d\omega_2\wedge d\omega_3\wedge \left(\frac{\partial \theta_3}{\partial k_2}\frac{\partial k_2}{\partial k_1}\frac{\partial k_1}{\partial y_1} d\omega_1\right).\\
\end{aligned}
\end{equation}
It remains to compute the parenthetical term. The second two partial derivatives are straightforward. The first partial derivative is computed by taking the partial derivative of the formula $\cot \theta_3 = f'(y_3) - k_2$ with respect to $k_2$ on both sides. 
We obtain as a result
\[
\frac{\partial \theta_3}{\partial k_2}\frac{\partial k_2}{\partial k_1}\frac{\partial k_1}{\partial y_1}= - \sin^2\theta_3\tan^2\theta_2 \tan^2\theta_1 f''(x_0+\omega_1) \, . 
\]

\noindent
(ii) For $x \in [0,1)$ and $\theta \in [0,\pi) \setminus \{\pi/2\}$, define $U(x, \theta) = \{\omega \in [- \e, \e] : f'(x + \omega) - \tan \theta \neq 0\}$. Note that $U(x, \theta)$ has full Lebesgue measure in $[- \e, \e]$ by (H1). We define
\begin{align*}
V = \{(\omega_1, \omega_2, \omega_3) \in [- \e, \e]^3 : \,  & \omega_1 \in U(x_0, \theta_0), \omega_2 \in U(x_1, \theta_1), \\ 
& \omega_3 \in U(x_2, \theta_2) , \text{ and } f''(x_0 + \omega_1) \neq 0\} \, .
\end{align*}
By (H1) and Fubini's Theorem, $V$ has full measure in $[- \e, \e]^3$, and it is clearly open 
and dense. To show $\det dH \ne 0$, 
we need $\theta_i \ne 0$ for $i=1,2,3$ on $V$. This  follows from the fact that for 
$\theta_{i-1} \neq \pi/2$, if $\omega_i \in U(x_{i-1}, \theta_{i-1})$ then $\theta_i \neq 0, \pi/2$.

\medskip \noindent
(iii) Given $(x_3, y_3, \theta_3)$, we solve for $(\omega_1, \omega_2, \omega_3)$
so that $H(\omega_1, \omega_2, \omega_3)=(x_3, y_3, \theta_3)$. Letting
$(x_i, y_i, \theta_i)$,  $i = 1,2$, be the intermediate images, we note that 
$y_2$ is uniquely determined by $x_3 = f(y_3) - y_2$, $\theta_2$ is determined
by $\cot \theta_3 =  f'(y_3) - \tan \theta_2$, as is $\theta_1$ once $\theta_2$ and $y_2$ are fixed. This in turn determines $f'(x_0 + \omega_1)$, but here uniqueness of
solutions breaks down.


Let $\omega_1^{(i)} \in [- \e, \e], i=1, \dots, n$, give the required value of $f'(x_0+\omega_1^{(i)})$. We observe that each $\omega_1^{(i)}$ determines uniquely  
$y_1^{(i)} = x_0 + \omega_1^{(i)}$, $x_1^{(i)}=f(y_1^{(i)})-y_0$,
$\omega_2^{(i)} = y_2-x_1^{(i)}$, $x_2^{(i)} = f(y_2) - y_1^{(i)}$, and
finally $\omega_3^{(i)} = y_3 - x_2^{(i)}$. 
Thus the number of $H$-preimages of any one point in $\P\T^2$ cannot exceed $n$.
Finally, we have $n \le 2$ for $\e$ small, and $n \le \# (C_\psi'')$ for $\e$ as large as $\frac12$.
\end{proof}

\begin{cor}\label{cor:densityForm} For any stationary probability $\hat \mu$ of $\hat{\mathcal X}$,
we have $\hat \mu(\T^2 \times \{\pi/2\}) =0$, and for any 
$(x_0,y_0,\theta_0)$ with $\theta_0 \ne \pi/2$ and any $(x_3, y_3,\theta_3) \in \P\T^2$, 
the density of $\hat P^{(3)}_{(x_0,y_0,\theta_0)}$ at $(x_3, y_3,\theta_3)$
is given by
\begin{equation} \label{density}
\frac{1}{(2\e)^3}  \left(\sum_{\omega_1 \in \mathcal E(x_3, y_3,\theta_3)}  
\frac{1}{|f''(x_0 + \omega_1)|} \right)   \frac{1}{\rho(x_3,y_3,\theta_3)}
\end{equation}
where
$$
\mathcal E(x_3, y_3,\theta_3) = \{\omega_1: \exists \omega_2, \omega_3
\mbox{ such that } H(\omega_1, \omega_2, \omega_3) = (x_3, y_3,\theta_3)\}
$$
and
$$
\rho(x,y,\theta) = \sin^2 (\theta) \left[f'(f(y) - x)  (f'(y) - \cot \theta) - 1 \right]^2\ .
$$
\end{cor}

\begin{proof} To show $\hat \mu(\T \times \{\pi/2\}) = 0$, it suffices to show that 
given any $x \in [0,1)$ and any $ \theta \in [0,\pi)$, 
$\nu^\e \{\omega \in [- \e, \e] : f'(x + \omega) = \tan \theta\} = 0$, and 
that is true because $C''_\psi$ is finite by (H1). The formula in (\ref{density}) 
follows immediately
from the proof of Lemma \ref{comp}, upon expressing $\tan^2(\theta_2) \tan^2(\theta_1)$
in terms of $(x_3,y_3,\theta_3)$ as was done in the proof of Lemma \ref{comp}(iii).
\end{proof}

\section{Proof of Theorem \ref{thm:lyapEst}}

The idea of our proof is as follows: Let $\hat \mu$ be any stationary probability
of the Markov chain $\hat{\mathcal X}$.
To estimate the integral in $\lambda(\hat\mu)$, we need 
to know the distribution of $\hat \mu$ in the $\theta$-direction. 
Given that the maps $F_\omega$
are strongly uniformly hyperbolic on a large part of the phase space 
with expanding directions well aligned with the $x$-axis  (see Remark 1),
one can expect that under $dF^N_{\uo}$ for large $N$,
$\hat \mu$ will be pushed toward a neighborhood of $\{\theta=0\}$
on much of $\T^2$,
and that is consistent with $\lambda^\e_1 \approx \log L$. This reasoning, however, 
is predicated on $\hat \mu$ not being concentrated, or stuck, on very small sets far away 
from $\{\theta \approx 0\}$, a scenario not immediately ruled out
 as the densities of transition probabilities are not bounded.

We address this issue directly by proving
 in Lemma \ref{lem:quantComp} an {\it a priori} bound 
on the extent to which $\hat \mu$-measure can be concentrated on (arbitrary) small sets. 
This bound is used in Lemma \ref{lem:angleMassEst} to estimate the $\hat \mu$-measure of the
set in $\P\T^2$ not yet attracted to $\{\theta=0\}$ in $N$ steps. The rest of the proof
consists of checking that these bounds are adequate for our purposes.


\medskip
In the rest of the proof, let $\hat \mu$ be an arbitrary invariant 
probability measure of $\hat{\mathcal X}$.

\smallskip
\begin{lem}\label{lem:quantComp} Let $A \subset \{\theta \in [\pi/4, 3 \pi/4]\}$ be a Borel subset of $\P\T^2$. Then for $L$ large enough,
\begin{align}\label{eq:hatMuBound}
\hat \mu(A) \leq \frac{\hat C}{L^{\frac14}} \bigg(1 + \frac{1}{\e^3 L^2} \Leb(A) \bigg) \, ,
\end{align}
for all $\e \in (0, \frac12],$ where $\hat C > 0$ is a constant independent of $L, \e$ or $A$.
\end{lem}

\begin{proof} 
By the stationarity of $\hat \mu$, we have, for every Borel set 
$A \subset \P\T^2$,
\begin{equation} \label{A}
\hat \mu(A) =  \int_{\P\T^2} \hat P^{(3)}_{(x_0, y_0, \theta_0)}(A) 
\, d \hat \mu(x_0, y_0, \theta_0)\ .
\end{equation}
Our plan is to decompose this integral into a main term and ``error terms", depending
on properties of the density of $\hat P^{(3)}_{(x_0, y_0, \theta_0)}$. The decomposition
is slightly different depending on whether $\e \le L^{-\frac12}$ or $\ge L^{-\frac12}$.

\medskip \noindent
{\bf The case $\e \le L^{-\frac12}$.} Let $K_2 \ge 1$ be such that $|\psi''(x)| \ge K_2^{-1} d(x,C''_\psi)$; such
a $K_2$ exists by (H1),(H2). Define $B'' = \{(x, y) : d(x,C_\psi'') \le 2K_2 L^{- 1/2} \}$.
Then splitting the right side of (\ref{A}) into
\begin{equation} \label{B}
\int_{B'' \times [0, \pi)} \  + \ \int_{\P \T^2 \setminus (B'' \times [0, \pi))} \ ,
\end{equation}
we see that the first integral is $\le \mbox{Leb}(B'') \le \frac{4K_2 M_2}{\sqrt L}$ where $M_2 = \# C_\psi''$. As for $(x_0,y_0) \not \in B''$, since $|f''(x_0+\omega)| \ge L^{\frac12}$, the density of 
$\hat P^{(3)}_{(x_0, y_0, \theta_0)}$ is $\le [(2\e)^3 M_2^{-1} L^{\frac12} \rho]^{-1}$
by Corollary \ref{cor:densityForm}.

To bound the second integral in (\ref{B}), we need to consider the zeros of $\rho$. As $A \subset \T^2 \times [\pi/4, 3 \pi/4]$, 
we have $\sin^2(\theta_3) \geq 1/2$. The form of $\rho$ in Corollary \ref{cor:densityForm} prompts us to decompose $A$ into 
$$
A = (A \cap \hat G) \cup (A \setminus \hat G)
$$ 
where $\hat G = G \times [0,\pi)$ and
$$
G = \{(x,y) : d(y, C'_\psi) > K_1 L^{-\frac12}, \ d(f(y)-x, C'_\psi) \ge K_1 L^{-\frac12}\}\ .
$$
Then on $\hat G \cap A$, we have
$\rho \ge \frac12 (\frac12 L)^2$ for $L$ sufficiently large. This gives
$$
\int_{\P \T^2 \setminus (B'' \times [0, \pi))} 
\hat P^{(3)}_{(x_0, y_0, \theta_0)}(A \cap \hat G) \, d \hat \mu \le 
\frac{C}{\e^3 \sqrt L} \frac{1}{L^2} \ \mbox{Leb}(A)\ .
$$
Finally, by the invariance of $\hat \mu$,
$$
\int_{\P \T^2 \setminus (B'' \times [0, \pi))} 
\hat P^{(3)}_{(x_0, y_0, \theta_0)}(A \setminus \hat G) \, d \hat \mu \ \le \ 
\hat \mu(A \setminus \hat G) = 
\mbox{Leb}(\T^2 \setminus G)\ .
$$
We claim that this is $ \lesssim L^{-\frac12}$. Clearly, Leb$\{d(y, C'_\psi) \le K_1 L^{-\frac12}\})
\approx L^{-\frac12}$. As for the second condition,
$$
\{y: f(y) \in (z-K_1 L^{-\frac12}, z+ K_1 L^{-\frac12})\}
= \{y: \psi(y) \in (z'-K_1 L^{-\frac32}, z'+ K_1 L^{-\frac32})\} 
$$
which in the worst case has Lebesgue measure $\lesssim L^{-\frac34}$ by (H1), (H2).

\medskip \noindent
{\bf The case $\e \ge L^{-\frac12}$.} Here we let $\tilde B'' = \{(x,y) : d(x,C_\psi'') \le K_2 L^{- 3/4}\}$, and decompose the right side of (\ref{A}) into
\[
\int (\hat P^{(3)}_{(x_0, y_0, \theta_0)})_1(A) \, d \hat \mu \ + \ 
\int (\hat P^{(3)}_{(x_0, y_0, \theta_0)})_2(A) \, d \hat \mu
\] 
where, in the notation in Sect. 3.2, 

\[
(\hat P^{(3)}_{(x_0, y_0, \theta_0)})_1 = 
H_* ((\nu^\e)^{ 3}|_{\{x_0 + \omega_1 \in \tilde B''\}}) \quad \mbox{and} \quad
(\hat P^{(3)}_{(x_0, y_0, \theta_0)})_2 = 
H_* ((\nu^\e)^{ 3}|_{\{x_0 + \omega_1 \notin \tilde B''\}}) \ .
\]
Then the first integral is bounded above by
$$
\sup_{x_0 \in \mathbb S^1} \nu^\e \{\omega_1 \in \tilde B'' - x_0\} 
\lesssim  \e^{-1} \Leb(\tilde B'') \leq Const \cdot L^{- 1/4}\ ,
$$
while the density of $(\hat P^{(3)}_{(x_0, y_0, \theta_0)})_2$ is $\le  [(2\e)^3 M_2^{-1} 
L^{1/4} \cdot \rho(x_3,y_3,\theta_3)]^{-1}$. The second integral is 
treated as in the case of $\e \le L^{-\frac12}$. 
\end{proof}

\smallskip
As discussed above, we now proceed to estimate
 the Lebesgue measure of the set that remains far
away from $\{\theta =0\}$ after $N$ steps, where $N$ is arbitrary for  now.
For fixed  $\uo = (\omega_1, \dots, \omega_N )$, we write
$(x_i,y_i) = F^i_{\uo} (x_0,y_0)$ for $1 \le i \le N$, and define $G_N=G_N(\omega_1, \dots, \omega_N )$ by
\begin{align*}
G_N = \{(x_0, y_0) \in \T^2 : & \, d(x_i + \omega_{i + 1}, C_\psi') \geq K_1 L^{-1 + \b} \text{ for all } 0 \leq i \leq N-1  \}\ .
\end{align*}
We remark that for $(x_0, y_0) \in G_N$, the orbit $F^i_{\uo} (x_0,y_0), i \le N$,
passes through uniformly hyperbolic regions of $\T^2$, where invariant
cones are preserved and $|f'(x_i + \omega_{i + 1})| \ge L^\beta$ for each $i<N$; see Remark 1 in Section 2.
We further define $\hat G_N = \{ (x_0,y_0,\theta_0) : (x_0,y_0) \in G_N\}$.

\smallskip
\begin{lem}\label{lem:angleMassEst} Let $\beta>0$ be given. We assume 
$L$ is sufficiently large (depending on $\b$).
Then for any $N \in \mathbb N$, $\e \in (0, \frac12]$ and 
$\omega_1, \dots, \omega_N  \in [- \e, \e]$, 
$$
\hat \mu (\hat G_N \cap \{|\tan \theta_N| > 1\}) \leq 
\frac{\hat C}{L^{\frac14}} \left(1+ \frac{1}{\e^3 L^{2+\beta N}} \right) \ .
$$
\end{lem}

\begin{proof} For $(x_0,y_0) \in G_N$, consider the singular value decomposition of 
$(dF_{\uo}^N)_{(x_0, y_0)}$. Let $\vartheta^-_0$ denote the angle corresponding to the most contracted direction of $(dF_{\uo}^N)_{(x_0, y_0)}$ and $\vartheta_N^-$ its
image under $(dF_{\uo}^N)_{(x_0, y_0)}$, and let $\sigma > 1 > \sigma^{-1}$ denote
the singular values of $(dF_{\uo}^N)_{(x_0, y_0)}$. A straightforward computation gives
\[
\frac12 L^\beta \leq |\tan \vartheta^-_0|, |\tan \vartheta_N^-| \quad \text{ and } 
\quad \sigma\geq \left( \frac13 L^\b  \right)^N \, .
\]
It follows immediately that for fixed $(x_0,y_0)$, 
$\{\theta_0 : |\tan \theta_N| >1\} \subset [\pi/4, 3\pi/4]$ and
$$
\mbox{Leb}\{\theta_0 : |\tan \theta_N| >1\} < \mbox{ const } L^{-\beta N}\ .
$$ 
Applying Lemma \ref{lem:quantComp} with $A = \hat G_N \cap \{ | \tan \theta_N| > 1\}$,
we obtain the asserted bound.
\end{proof}

\medskip
By the stationarity of $\hat \mu$, it is true for any $N$ that
$$
\lambda(\hat \mu) = \int \left( \int \| (dF_{\omega_{N + 1}})_{(x_N, y_N) }u_{\theta_N} \|
\ d (\hat F_{\omega_N} \circ \dots \circ \hat F_{\omega_1})_* \hat \mu \right)
d\nu^\e(\omega_1) \cdots d \nu^\e(\omega_{N + 1}) \, .
$$
We have chosen to estimate $\lambda(\hat \mu)$ one sample path at a time because  
we have information from Lemma \ref{lem:angleMassEst} on    
$(\hat F_{\omega_N} \circ \dots \circ \hat F_{\omega_1})_*\hat \mu$ for each sequence $\omega_1, \dots, \omega_N$.

\medskip
\begin{prop}\label{prop:consLE}
Let $\alpha, \b \in (0,1)$. Then, there are constants $C = C_{\alpha, \beta} > 0$
and $C' = C_{\alpha, \beta}' > 0$ such that for any $L$ sufficiently large, we have the following. Let $N = \lfloor C' L^{1 - \b} \rfloor$, $\e \in [ L^{- C L^{1 - \beta}}, \frac12]$, and fix arbitrary $\omega_1, \cdots, \omega_{N + 1} \in [- \e, \e]$. Then,
\begin{align}\label{main}
I \ := \ \int_{\P \T^2} \log \| (d F_{\omega_{N + 1}})_{(x_N, y_N) }u_{\theta_N} \|  \, d \hat \mu(x_0,y_0, \theta_0) \ \geq \ \a \log L \, .
\end{align}
\end{prop}

Integrating (\ref{main}) over $(\omega_1, \dots, \omega_{N + 1})$
gives $\lambda(\hat \mu) \ge \a \log L$. 
As $\lambda^\e_1 \ge \lambda(\hat \mu)$, part (b) of Theorem \ref{thm:lyapEst} follows
immediately from this proposition.

\smallskip
\begin{proof} The number $N$ will be determined in the course of the proof, 
and $L$ will be enlarged a finite number of times as we go along. 
As usual, we will split $I$, the integral in (\ref{main}), to one on a good and 
a bad set.
The good set is essentially the one in Lemma \ref{lem:angleMassEst}, with an additional condition on 
$(x_N,y_N)$, where $dF$ will be evaluated. Let 
\begin{align*}
G^*_N = \{(x_0, y_0) \in G_N : d(x_N + \omega_{N + 1}, C_\psi') \geq K_1 m \} \, ,
\end{align*}
where $m > 0$ is a small parameter to be specified later. 
As before, we let $\hat G^*_N = G^*_N \times [0,2\pi)$.
Then $\mathcal G:= \hat G^*_N \cap \{|\tan \theta_N| \leq 1\}$ is the good set; 
on $\mathcal G$, the integrand in (\ref{main}) is $\geq \log \big( \frac{m L}{4} \big)$. Elsewhere we use the worst
 lower bound $- \log (2 \|\psi'\|_{C_0} L)$. Altogether we have
 \begin{align}\label{eq:intermediate}
I \ge \ \log (\frac14 m L) - \log \frac{m \|\psi'\| L^2}{2}  \ \hat \mu (\mathcal B) \ ,
\end{align}
where 
\begin{equation} \label{bad}
\mathcal B = \P\T^2 \setminus \mathcal G =
(\P\T^2 \setminus \hat G^*_N) \cup (\hat G^*_N 
\cap \{|\tan \theta_N| > 1\})\ .
\end{equation}
We now bound $\hat \mu(\mathcal B)$. First, 
\begin{equation} \label{error1}
\hat \mu(\P\T^2 \setminus \hat G^*_N) = 1- \Leb(G^*_N) \le K_1 M_1 (m + NL^{-1 + \b}) \, ,
\end{equation}
where $M_1 = \# C_\psi'$. Letting $N= \lfloor C'L^{1-\b} \rfloor $ and $m=C'=\frac{p}{4K_1M_1}$
where $p$ is a small number to be determined,
we obtain $\hat \mu(\P\T^2 \setminus \hat G^*_N) \le \frac12 p$. 
From Lemma \ref{lem:angleMassEst}, 
\begin{equation} \label{error2}
\hat \mu(\hat G^*_N \cap \{|\tan \theta_N| > 1\}) \le 
\frac{\hat C}{L^{\frac14}} \bigg(1 + \frac{ 1}{(\e L^{\frac13 \b N})^3} \bigg)\ .
\end{equation}
For $N$ as above and $\e$ in the designated range (with $C = \frac \b 3 C'$), 
the right side of (\ref{error2}) is easily made $< \frac12 p$ 
by taking $L$ large, so we have $\hat \mu(\mathcal B) \le p$. 
Plugging into \eqref{eq:intermediate}, we see that
$$
I \ \ge \ (1 - 2p) \log L\ - \{\mbox{ terms involving } \log p, \ p\log p \ 
\mbox{ and constants }\}\ .
$$
Setting $p = \frac14 (1 - \a)$ and taking $L$ large enough, one ensures that
$I > \a \log L$.
\end{proof}


\section{Proof of Theorem \ref{thm:lyapEst2}}\label{Sb=1o}

We now show that with the additional assumption (H3), the same result holds
for $\e \ge L^{- C L^{-2+\b}}$.

\subsection{Proof of theorem modulo main proposition}

As the idea of the proof of Theorem \ref{thm:lyapEst2} closely parallels that of Proposition \ref{prop:consLE}, it is useful to recapitulate the main ideas:
\begin{itemize}
\item[1.] The main Lyapunov exponent estimate is carried on the subset
$\{(x_0,y_0,\theta_0): (x_0,y_0) \in G_N,
|\tan \theta_N|<1\}$ of $\P\T^2$, where $G_N$
consists of points whose orbits stay 
 $\gtrsim L^{-1+\b}$ away from $C'_\psi \times \mathbb S^1$ in their first $N$ iterates.
\item[2.] Since Leb$(G^c_N)\sim NL^{-1+\b}$, we must take $N \lesssim L^{1-\b}$.
\item[3.] By the uniform hyperbolicity of $F^N_{\uo}$ on $G_N$, Leb$\{|\tan \theta_N|>1\} 
\sim L^{-cN}$.
\item[4.] For $\hat \mu\{|\tan \theta_N|>1\}$ to be small, we must have $\frac{1}{\e^3} L^{-cN} \ll 1$ (Lemma \ref{lem:angleMassEst}).
\end{itemize} 
Items 2--4 together suggest that we require $\e > L^{-\frac13 cN} \ge 
L^{-c'L^{-1+\b}}$, and we checked that for this $\e$ the proof goes through.

\bigskip
The proof of Theorem \ref{thm:lyapEst2} we now present differs from the above in the following way:
The set $G_N$, which plays the same role as in Theorem \ref{thm:lyapEst}, will be different.
It will satisfy 
\begin{itemize}
\item[(A)] Leb$(G^c_N)\sim NL^{-2+\b}$, and
\item[(B)] the composite map $dF^N_{\uo}$ is uniformly hyperbolic on $G_N$.
\end{itemize}
The idea is as follows: To decrease $\e$, we must increase $N$, while keeping
the set $G^c_N$ small. This can be done by allowing the random orbit to come closer to
$C_\psi' \times \mathbb S^1$, but with that, one cannot expect uniform hyperbolicity
in each of the first $N$ iterations, so we require only (B). This is the main difference
between Theorems \ref{thm:lyapEst} and \ref{thm:lyapEst2}.
Once $G_N$ is properly identified and properties (A) and (B) are proved, the rest of
the proof follows that of Theorem \ref{thm:lyapEst}: Property (A) permits us to take 
$N \sim L^{2-\b}$ in item 2, and item 3 is valid by Property (B).
Item 4 is general and therefore unchanged, leading
to the conclusion that it suffices to assume $\e > L^{-c'L^{-2+\b}}$.
As the arguments follow those in Theorem \ref{thm:lyapEst} {\it verbatim} modulo the bounds above and 
accompanying constants, we will not repeat
the proof. The rest of this section is focused on
producing $G_N$ with the required properties.

\bigskip \noindent
{\it It is assumed from here on} that (H3)$(c_0)$ holds,
and $L, a$ and $\e$ are as in Theorem \ref{thm:lyapEst2}. Having proved Theorem \ref{thm:lyapEst}, 
we may assume $\e \le L^{-1}$. In light of the discussion above,
$\omega_1, \cdots, \omega_{N} , \omega_{N + 1} \in [-\e,\e]$ will be fixed throughout, and
$(x_i,y_i)=F^i_{\uo}(x_0,y_0)$ as before. 

\bigskip \noindent
{\bf Definition of $G_N$.} For arbitrary $N$ we define $G_N$ to be
\begin{align*}
G_N = \{(x_0, y_0) \in \T^2 : 
(a) & \text{ for all } 0 \leq i \leq N -1, \, \\
&   \qquad (i) \ d(x_i + \omega_{i + 1} , C_\psi') \geq K_1 L^{-2 + \b} , \\
&  \qquad (ii) \ d(x_i + \omega_{i + 1} , C_\psi')\cdot d(x_{i+1} + \omega_{i + 2} , C_\psi') \geq K_1^{2} L^{-2+ \b / 2} , \\
(b) & \ d(x_0 + \omega_1, C_\psi'), \ d(x_{N-1} + \omega_{N}, C_\psi')\geq p/(16M_1)
\} 
\end{align*}
where $M_1 = \# C_\psi'$ and $p=p(\alpha)$ is a small number to be determined. 
Notice that (a)(i)  implies only $|f'(x_i+\omega_{i+1})| \geq L^{-1 + \b}$,
not enough to guarantee expansion in the horizontal direction. We remark also that even though (a)(ii)
implies $|f'(x_i + \omega_{i + 1}) f'(x_{i+1} + \omega_{i + 2})| \geq L^{\b / 2}$,
hyperbolicity does not follow without control of the {\it angles} of the vectors involved. 

\smallskip
\begin{lem}[Property (A)] There exists $C_2 \ge 1$ such that for all $N$, 
$$
\mbox{Leb}(G_N^c) \le C_2 NL^{-2+\b} + \frac{p}{4}\ .
$$
\end{lem}

\begin{proof}
Let 
\begin{align*}
A_1 &= \{x \in [0,1) : d(x, C_\psi') \geq K_1 L^{-2 + \b} \} \, , \\
A_2 &= \{(x, y) \in \T^2 : x \in A_1, \text{ and } d(x, C_\psi') \cdot d(f x - y, C_\psi') \geq K_1^2 L^{-2 + \b/2} \} \, .
\end{align*}
We begin by estimating $\Leb(A_2)$. Note that $\Leb(A_1^c) \leq 2 M_1 K_1 L^{-2 + \b}$, and for each fixed $x \in A_1$, 
\begin{equation}
\Leb \left\{y \in [0,1) : d(f x - y, C_\psi') < \frac{K_1^2 L^{-2 + \b/2}}{d(x, C_\psi')} \right\} 
\leq \frac{2 M_1 K_1^2 L^{-2 + \b/2}}{d(x, C_\psi')} \, ,
\end{equation}
hence
\begin{align*}
\Leb A_2^c & \leq \Leb A_1^c + \int_{x \in A_1} \frac{2 M_1 K_1^2 L^{-2 + \b/2}}{d(x, C_\psi')} dx \, .
\end{align*}
Let $\hat c = \frac12\min\{ d(\hat x,\hat x')  : \hat x, \hat x' \in C_\psi', \hat x \neq \hat x'\}$ . We split the integral above into $\int_{d(x, C_\psi') > \hat c}$ $+ \int_{K_1 L^{-2 + \b} \leq d(x, C_\psi') \leq \hat c}$. The first one is bounded from above by $2 M_1 K_1^2 \hat c^{-1} L^{-2 + \b/2}$, and the second by
\begin{equation}
 4 M_1^2 K_1^2 L^{-2 + \b/2} { \int_{ K_1 L^{-2 + \b} }^{\hat c} \frac{du}{u} }
 \leq 4 (2 - \b) M_1^2 K_1^2 L^{-2 + \b/2}  \log L   \, ,
\end{equation}
(having used that $- \log K_1$ and $\log \hat c$ are $< 0$) and so on taking $L$ large enough so that $L^{\b/2} \geq \log L$, it follows that $\Leb(A_2^c) \leq C_2 L^{-2 + \b}$, where $C_2 = C_{2, \psi}$ depends on $\psi$ alone.

Let $\tilde G_N$ be equal to $G_N$ with condition (b) removed. Then
$$
\tilde G_N = \bigcap_{i = 0}^{N-1}  (F_{\uo}^i)^{-1} \big(  A_2 - (\omega_{i + 1}, 0) \big)\ ,
$$
so $\Leb(\tilde G_N) \geq 1 - C_2 N L^{-2 + \b}$. The rest is obvious.
\end{proof}

\smallskip
\begin{prop}[Property (B)] \label{prop:coneEstimate2} For any $N \ge 2$, 
$(dF^N_{\uo})_{(x_0, y_0)}$ is hyperbolic on $G_N$ with the following uniform bounds:
The larger singular value $\sigma_1$ of $(dF^N_{\uo})_{(x_0, y_0)}$ satisfies
\[
\sigma_1\big(  (dF^N_{\uo})_{(x_0, y_0)} \big) \geq L^{\frac{\b}{15} N } \, ,
\]
and if $\vartheta_0^- \in [0,\pi)$ denotes the most contracting direction of $(dF^N_{\uo})_{(x_0, y_0)}$ and $\vartheta_N^- \in [0,\pi)$ its image, then 
\[
|\vartheta_0^- - \pi/2|, |\vartheta_N^- - \pi/2| \leq L^{- \b} \, .
\]
\end{prop}

\medskip
The bulk of the work in the proof of Theorem \ref{thm:lyapEst2} goes into proving this proposition.


\subsection{Proof of Property (B) modulo technical estimates}

 Let $c =c_\psi \ll c_0$ where $c_0$ is as in (H3); we stipulate additionally that $c \leq p / 16 M_1$, where $p=  p_\alpha$ and $M_1$ are as before. First we introduce the following symbolic encoding of $\T^2$. Let
$$
 B = {\mathcal N}_{\sqrt{\frac{c}{L}}}( C'_\psi) \times \mathbb S^1, \quad 
I = {\mathcal N}_c( C'_\psi) \times \mathbb S^1 \setminus B, 
\quad \mbox{and} \quad G = \T^2 \setminus (B \cup I)\ .
$$
To each $(x_0, y_0) \in \T^2$ we associate a symbolic sequence 
$$(x_0,y_0) \ \mapsto \ \bar W = W_{N-1} \cdots W_1 W_0 \in \{B, I, G\}^N\ ,
$$
where 
$(x_i + \omega_{i + 1}, y_i) \in W_i$. We will refer to any symbolic sequence of length $\ge 1$, 
e.g. $\bar V = GBBG$, as a {\it word}, and use Len$(\bar V)$ to denote the length
of $\bar V$, i.e., the number of letters it contains. We also write
$G^k$ as shorthand for a word consisting
of $k$ copies of $G$. Notice that symbolic sequences are to be read from right
to left.

The following is a direct consequence of (H3). 

\begin{lem}\label{lem:badToGood} Assume that $\e < L^{-1}$. Let $(x_0,y_0) \in \T^2$ be such that $(x_0 + \omega_1, y_0) \in B \cup I$ and $(x_1 + \omega_2, y_1) \in B$.
Then $(x_2 + \omega_3, y_2) \in G$. 
\end{lem}

\begin{proof} Let $\hat x_0, \hat x_1 \in C'_\psi$ (possibly $\hat x_0 = \hat x_1$) 
 be such that $d(x_0,\hat x_0 )< c$ and $d(x_1,\hat x_1) < \sqrt{\frac{c}{L}}$.
Since $(f(\hat x_1) - \hat x_0) \mbox{ (mod 1) } \not \in \mathcal N_{c_0}( C'_\psi)$
by (H3), it suffices to show $|x_2 - (f(\hat x_1) - \hat x_0) \mbox{ (mod 1)}| \ll c_0$:
\begin{eqnarray*}
|x_2 - (f(\hat x_1) - \hat x_0) \mbox{ (mod 1)}| & = &
|(f(x_1 + \omega_2)-y_1) - (f(\hat x_1) - \hat x_0) \mbox{ (mod 1)}|\\
& \le & | f(x_1 + \omega_2) -f(\hat x_1)| + d(y_1, \hat x_0)\ .
\end{eqnarray*} 
To see that this is $\ll c_0$, observe that for large $L$, we have
$$
|f(x_1 + \omega_2) - f(\hat x_1)| < \frac12 L \|\psi''\| \left(\sqrt{\frac{c}{L}} + 
L^{-1} \right)^2 < \|\psi''\| c\ ,
$$
and $d(y_1, \hat x_0) = d(x_0+\omega_1, \hat x_0) < 2c$.
\end{proof}


Next we apply Lemma \ref{lem:badToGood} to put constraints on the set of all possible words $\bar W$ associated with $(x_0, y_0) \in G_N$.

\begin{lem}
Let $\bar W$ be associated with  $(x_0, y_0) \in G_N$. Then
$\bar W$ must have the following form:
\begin{align}\label{eq:wn}
\bar W = G^{k_M} \bar V_M G^{k_{M-1}} \bar V_{M-1}  \cdots G^{k_1} \bar 
V_1 G^{k_0} \, ,
\end{align}
where $M \geq 0$, $k_0, k_1, \cdots, k_M \geq 1$, and if $M > 0$, then each
 $\bar V_i$ is one of the words in
$$
{\mathcal V} \ = \ \{B, BB, \   \mbox{ or } \ B I^k B, I^k B, I^k, B I^k \ \mbox{ for some }k \geq 1 \}\ .
$$
\end{lem}

\begin{proof}
The sequence $\bar W$ starts and ends with $G$ by the definition of $G_N$ and the stipulation that $c \leq p / (16 M_1)$; thus a decomposition of the form \eqref{eq:wn} is obtained with words $\{\bar V_i\}_{i = 1}^M$ formed from the letters $\{I, B\}$. To show that the words $\{\bar V_i\}_{i = 1}^M$ must be of the proscribed form, observe that
\begin{itemize}
\item $BB$ occurs only as a subword of $GBBG$;
\item  $BI$ only occurs as a subword of $GBI$;
\item  $IB$ only occurs as a subword of $IBG$.
\end{itemize}
Each of these constraints follows from Lemma \ref{lem:badToGood}; for the third,
$G$ is the only letter that can precede $IB$. It follows from the last two bullets that all 
the $I$s must be consecutive, and $B$ can appear at most twice.
\end{proof}

With respect to the representation in (\ref{eq:wn}), we view each $\bar V_i$ as representing
an excursion away from the ``good region" $G$. In what follows, 
we will show that $G_N$ and (H3) are chosen so that
for $(x_0,y_0) \in G_N$, vectors are not rotated by too much during these excursions, 
and hyperbolicity is restored with each visit to $G$. 
To prove this, we introduce the following cones in tangent space:
\[
{\mathcal C}_n = {\mathcal C}(L^{-1 + \b/4}) \, , \qquad {\mathcal C}_1 = {\mathcal C}(1) \, , \quad \text{and } \quad {\mathcal C}_w = {\mathcal C}(L^{1 - \b/4}) \, ,
\]
where ${\mathcal C}(s)$ refers to the cone of vectors whose slopes have absolute value $\le s$. The letters $n, w$ stand for `narrow' and `wide', respectively.

Let $(x_0,y_0) \in G_N$ and suppose for some $m$ and $l$, 
$\{(x_{m+i-1} + \omega_{m+i}, y_{m+i-1})\}_{i=1}^l$
corresponds to the word $\bar V = V_l \cdots V_1 \in \mathcal V$. 
To simplify notation, we write 
$$
(\tilde x_i, \tilde y_i)= (x_{m+i-1} + \omega_{m+i}, y_{m+i-1}) \qquad \mbox{and} \qquad
d{\tilde F}^l = dF_{(\tilde x_l, \tilde y_l)} \circ \cdots \circ dF_{(\tilde x_1, \tilde y_1)}\ .
$$

\begin{prop}\label{lem:allWords}
Let $\{(\tilde x_i, \tilde y_i)\}_{i = 1}^l $ and $\bar V = V_l \cdots V_1\in \mathcal V$
be as above. Then
$$
d\tilde F^l ({\mathcal C}_n)  \subset {\mathcal C}_w\ ,
\quad \big(d\tilde F^l) ^* ({\mathcal C}_n)  \subset {\mathcal C}_w\ , \quad
\mbox{and} \quad
\min_{u \in {\mathcal C}_n,  \|u \| = 1} \| d\tilde F^l u\| \ \geq \ 
\frac12 L^{\frac{\b}{5} n_I (\bar V)}
$$
where $n_I(\bar V)$ is the number of appearances of the letter $I$ in the word $\bar V$. 
\end{prop}

We defer the proof of Proposition \ref{lem:allWords} to the next subsection.

\begin{proof}[Proof of Proposition \ref{prop:coneEstimate2} assuming Proposition \ref{lem:allWords}]
For $(x_0, y_0) \in G_N$, let $\bar W$ be as in \eqref{eq:wn}. It is easy to check
that if $(x_{m} + \omega_{m + 1}, y_{m}) \in G$, then
\begin{equation} \label{hypbound}
(dF_{\omega_{m + 1}})_{(x_{m} ,y_{m})}  ({\mathcal C}_w) \subset {\mathcal C}_n
\quad \mbox{with} \quad
\min_{u \in \Cc_w, \|u\| = 1} \|(dF_{\omega_{m + 1}})_{(x_{m} ,y_{m})}  u\| \geq \frac14 L^{\b/4} \, .
\end{equation}
Applying (\ref{hypbound}) and Proposition \ref{lem:allWords} alternately, we obtain
\[
(dF^N_{\uo})_{(x_0, y_0)} (\Cc_w) \subset \Cc_n \, .
\]
Identical considerations for the adjoint yield the cones relation $(dF^N_{\uo})_{(x_0, y_0)}^* \Cc_w \subset \Cc_n$. We now use the following elementary fact from linear algebra: if $M$ is a $2 \times 2$ real matrix with distinct real eigenvalues $\eta_1 >  \eta_2$ and corresponding eigenvectors $v_1, v_2 \in \R^2$, and if ${\mathcal C}$ is any closed convex cone with nonempty interior for which $M {\mathcal C} \subset {\mathcal C}$, then $v_1 \in {\mathcal C}$. 

We therefore conclude that the maximal expanding direction $\vartheta_0^+$ for $(dF^N_{\uo})_{(x_0, y_0)}$ and its image $\vartheta_N^+$ both belong to ${\mathcal C}_n$. The estimates for $\vartheta_0^-, \vartheta_N^-$ now follow on recalling that $\vartheta_0^- = \vartheta_0^+ + \pi/2 \, (\text{mod } \pi), \vartheta_N^- = \vartheta_N^+ + \pi/2 \, (\text{mod } \pi)$.

It remains to compute $\sigma_1 \big( (dF^N_{\uo})_{(x_0, y_0)} \big)$. From 
(\ref{hypbound}) and the derivative bound in Proposition \ref{lem:allWords} gives
\begin{align*}
\min_{u \in \Cc_w, \|u\| = 1} \|(dF^N_{\uo})_{(x_0, y_0)} u\| \geq L^{\frac{\b}{5} \big( (k_0 - 1) + (k_1 - 1) + \cdots + (k_{M-1} - 1) + k_M + \sum_{i = 1}^M (n_I(\bar V_i) + 1) \big) }
\end{align*}
As there cannot be more than two copies of $B$ in each $\bar V \in \mathcal V$, 
we have
\[
\frac{n_I(\bar V) + 1}{\Len(\bar V) + 1} \geq \frac13 \, ,
\]
and the asserted bound follows.
\end{proof}

\subsection{Proof of Proposition \ref{lem:allWords}}

Cones relations for adjoints are identical to those of the original, and so are omitted: hereafter we work exclusively with the original (unadjointed) derivatives. We will continue to use the notation in Proposition \ref{lem:allWords}. Additionally, 
 in each of the assertions below, if $d\tilde F^l$ is applied to the cone $\Cc$, then $\min$ refers to the minimum
taken over all unit vectors $u \in \Cc$.


The proof consists of enumerating all cases of $\bar V \in \mathcal V$.
We group the estimates as follows:

\begin{lem}\label{lem:oneLetter}
\ 
\begin{itemize}
\item[(a)] For $\bar V = I$: \quad 
$d\tilde F (\Cc_1) \subset \Cc_1$ \quad and \quad $ 
\min \| d\tilde F u\| \geq \frac12 K_1 \sqrt c \sqrt L \gg L^{1/4}$.
\item[(b)] For $\bar V = B$: \quad 
$d\tilde F (\Cc_n) \subset \Cc_w$ \quad and \quad $ 
\min \| d\tilde F u\| \geq \frac12$.
\end{itemize}
\end{lem}

The next group consists of two-letter words the treatment of which will rely on 
condition (a)(ii) in the definition of $G_N$. 

\begin{lem}\label{lem:twoLetter}
\ 
\begin{itemize}
\item[(c)] For $\bar V = BB$: \quad
$d\tilde F^2(\Cc_n) \subset \Cc_w$ \quad and \quad
$\min \|d\tilde F^2u \| \geq L^{\b/3} $.
\item[(d)] For $\bar V = BI$: \quad 
$d\tilde F^2(\Cc_1) \subset \Cc_w$ \quad and \quad 
$\min \|d\tilde F^2 u\| \geq \min\{\frac12 K_1 \sqrt c \sqrt L , L^{\b/3}\} \geq L^{\b/5}$.
\item[(e)] For $\bar V = IB$: \quad 
$d \tilde F^2 (\Cc_n) \subset \Cc_1$ \quad and \quad 
$\min \|d\tilde F^2u \| \geq L^{\b/3} $.
\end{itemize}
\end{lem}

This leaves us with the following most problematic case:
\begin{lem}\label{lem:special}
\ 
\begin{itemize}
\item[(f)] For $\bar V = BIB$: \quad
$d\tilde F^3 (\Cc_n) \subset \Cc_w$ \quad and \quad
$\min \|d\tilde F^3 u\| \geq L^{\b/5}$\ .
\end{itemize}
\end{lem}

\smallskip
\begin{proof}[Proof of Proposition \ref{lem:allWords} assuming Lemmas \ref{lem:oneLetter}--\ref{lem:special}]
We go over the following checklist:
%
%
\begin{itemize}
\item $\bar V = B$ or $BB$ was covered by (b) and (c); total growth on $\Cc_n$ is $\geq \frac12$.
\end{itemize}
For $k \ge 1$, 
\begin{itemize}
\item $\bar V = I^k$ follows from (a); total growth on $\Cc_n$ is $\geq L^{k /4} \gg L^{\frac{\b}{5} k}$.
\item $\bar V = I^k B = I^{k-1}(IB)$ follows from concatenating (e) and (a); total growth on $\Cc_n$ is 
$\geq L^{(k-1) / 4} \cdot L^{\b/3} \gg L^{\frac{\b}{5} k}$.
\item $\bar V = B I^k = (BI)I^{k-1}$ follows from concatenating (a) and (d); total growth on $\Cc_n$ is 
$\geq L^{\b/5} \cdot L^{(k-1)/4} \gg L^{\frac{\b}{5} k}$.
\end{itemize}
Lastly, 
\begin{itemize}
\item $\bar V = BIB$ follows from (f); total growth on $\Cc_n$ is $ \geq L^{\b/5}$,
and
\item for $k \ge 2$, $\bar V = BI^kB = (BI) I^{k-2} (IB)$ follows by concatenating (e), followed
by (a) then
 (d); total growth on $\Cc_n$ is $\geq L^{\b/5} \cdot L^{(k-2) /4} \cdot L^{\b/3} \gg L^{\frac{\b}{5} k}$.
\end{itemize}
This completes the proof.
\end{proof}

Lemma \ref{lem:oneLetter} is easy and left to the reader; it is a straightforward application of the formulae
\[
\tan \theta_{ 1} = \frac{1}{f'(\tilde x_1) - \tan \theta_0} \, , \quad \| d \tilde F u_\theta \| = \sqrt{(f'(\tilde x_1) \cos \theta_0 - \sin \theta_0)^2 + \cos^2 \theta_0} \, ,
\]
where $\theta_1 \in [0, \pi)$ denotes the angle of the image vector $d \tilde F u_{\theta_0}$.

Below we let $K$ be such that $|f'| \le KL$.

\begin{proof}[Proof of Lemma \ref{lem:twoLetter}]
We write $u = u_{\theta_0}$ and $\theta_1, \theta_2 \in [0,2 \pi)$ for the angles of the images $d\tilde F u, d\tilde F^2u$ respectively. Throughout, we use the following `two step' formulae:
\begin{gather}
\label{eq:twoStepAngle} \tan \theta_2 = \frac{f'(\tilde x_1) - \tan \theta_0}{f'(\tilde x_1) f'(\tilde x_{2}) - f'(\tilde x_{2}) \tan \theta_0 - 1} \, , \\
\label{eq:twoStepGrowth} \| d\tilde F^2 u_\theta \|  \geq |( f'(\tilde x_1)  f'(\tilde x_2) - 1) \cos \theta_0| - | f'(\tilde x_2) \sin \theta_0| 
\end{gather}
The estimate $|f'(\tilde x_1) f'(\tilde x_2)| \geq L^{\b / 2}$ (condition (a)(ii) in the definition
of $G_N$) will be used repeatedly throughout.

\medskip

We first handle the vector growth estimates. For (c) and (e), as $u = u_{\theta_0} \in \Cc_n$, the right side of (\ref{eq:twoStepGrowth}) is
$\geq \frac12 L^{\b/2} - 2 K L^{\b/4} \gg L^{\b/3}$.
%
%
For (d) we break into the cases 

(d.i) $|f'(\tilde x_2)| \geq L^{\b/4}$ and 

(d.ii) $|f'(\tilde x_2)| < L^{\b/4}$. 

\noindent In case (d.i), by (a) we have that $u_{\theta_1} \in \Cc_1$ and $\|d F_{(\tilde x_1, \tilde y_1)} u_{\theta_0}\| \geq \frac12 K_1 \sqrt c \sqrt L$. Thus $|\tan \theta_2| \leq 2 L^{- \b/4} \ll 1$ and $\| dF_{(\tilde x_2, \tilde y_2)} u_{\theta_1}\| \geq \frac12 L^{\b/4} \gg 1$, completing the proof. In case (d.ii), the right side of (\ref{eq:twoStepGrowth}) is
\[
\geq \frac{1}{\sqrt 2} (L^{\b/2} - 1) - \frac{1}{\sqrt 2} L^{\b/4} \gg L^{\b/3} \, .
\]

We now check the cones relations for (c) -- (e). For (c), 
\[
|\tan \theta_2| \leq \frac{|f'(\tilde x_{1})| + |\tan \theta_0|}{|f'(\tilde x_1) f'(\tilde x_2) | - |f'(\tilde x_{2}) \tan \theta_0| - 1} 
\leq \frac{K L + L^{-1 +\b/4 }}{L^{\b/2} - K L^{ \b/4} - 1} \leq 2 K L^{1 - \b/2} \ll L^{1 - \b/4} \, , 
\]
so that $u_{\theta_2} \in \Cc_w$ as advertised. The case (d.i) has already been
treated. For (d.ii), the same bound as in (c) gives
\[
|\tan \theta_2| \leq  \frac{K L + 1 }{L^{\b/2} - L^{\b/4} - 1} \leq  2 K L^{1 - \b/2} \ll L^{1 - \b/4} \, ,
\]
hence $u_{\theta_2} \in \Cc_w$.

For (e) we again distinguish the cases 

(e.i) $|f'(\tilde x_1)| \geq L^{\b / 4}$ and 

(e.ii) $|f'(\tilde x_1)| < L^{\b / 4}$. 

\noindent In case (e.i), one easily checks that $dF_{(\tilde x_1, \tilde y_1)} (\Cc_n) \subset \Cc_1$ and then $dF_{(\tilde x_2, \tilde y_2)} (\Cc_1) \subset \Cc_1$ by (a). In case (e.ii) we compute directly that
\[
|\tan \theta_2| \leq \frac{|f'(x_{1})| + |\tan \theta_0|}{|f'(x_1) f'(x_2) | - |f'(x_{ 2}) \tan \theta_0| - 1} \leq \frac{L^{\b/4} + L^{-1 + \b/4}}{L^{\b/2} - K L^{\b/4} - 1} \ll 1 \, ,
\]
hence $u_{\theta_2} \in \Cc_1$.
\end{proof}

\smallskip
\begin{proof}[Proof of Lemma \ref{lem:special}.]
We let $u = u_{\theta_0} \in \Cc_n$ (i.e. $|\tan \theta_0| \leq L^{-1 + \b/4}$) and write $\theta_1, \theta_2, \theta_3 \in [0,\pi)$ for the angles associated to the subsequent images of $u$. We break into two cases: 

(I) $|f'(\tilde x_3)| \geq |f'(\tilde x_1)|$ and 

(II) $|f'(\tilde x_3)| < |f'(\tilde x_1)|$.

\smallskip
In case (I), we compute
\[
|\tan \theta_2 |\leq \frac{|f'(\tilde x_1)| + |\tan \theta_0|}{|f'(\tilde x_1) f'(\tilde x_2)| - |f'(\tilde x_2) \tan \theta_0| - 1} \leq \frac{2 |f'(\tilde x_1)|}{L^{\b/2} - 2 K L^{ \b/4} - 1} \leq 4 |f'(\tilde x_1)| L^{- \b/2} \, ,
\]
having used that $|f'(\tilde x_1)| \geq L^{-1 + \b}$ and $|\tan \theta_0| \leq L^{-1 + \b/4}$ in the second inequality. Now,
\[
|\tan \theta_3| \leq \frac{1}{|f'(\tilde x_3)| - |\tan \theta_2|} \leq \frac{1}{|f'(\tilde x_1)| - 4 |f'(\tilde x_1)| L^{- \b/2}} \leq \frac{2}{|f'(\tilde x_1)|} \leq 2 L^{1 - \b} \ll L^{1 - \b/4} \, .
\]
In case (II), we use
\[
|\tan \theta_1| \leq \frac{1}{|f'(\tilde x_1)| - |\tan \theta_0|} \leq \frac{1}{|f'(\tilde x_3)| - L^{-1 + \b/4}} \leq \frac{2}{|f'(\tilde x_3)|} \, ,
\]
again using that $|f'(\tilde x_3)| \geq L^{-1 + \b}$, and then
\begin{eqnarray*}
|\tan \theta_3| & \leq & \frac{|f'(\tilde x_2)| + |\tan \theta_1|}{|f'(\tilde x_2) f'(\tilde x_3)| - |f'(\tilde x_3) \tan \theta_1| - 1} \\
& \leq &
\frac{K L + 2 |f'(\tilde x_3)|^{-1}}{L^{\b/2} - 3} \leq \frac{K L + 2 L^{1 - \b}}{L^{\b/2} - 3} \leq 2 K L^{1 - \b/2} \ll L^{1 - \b/4} \, .
\end{eqnarray*}

For vector growth, observe that from (e) we have
$d\tilde F^2(\Cc_n) \subset \Cc_1$ and $\min \|d\tilde F^2 u\| \geq L^{\b/3}$.
So, if $|f'(\tilde x_3)| \geq L^{\b/12}$ then 
\[
\|dF_{(\tilde x_3, \tilde y_3)} u_{\theta_2}\| \geq \frac{1}{\sqrt{2}} (L^{\b/12} -1 ) \gg 1 \, .
\] 
Conversely, if $|f'(\tilde x_3)| < L^{\b/12}$ then we can use the crude estimate $\| (dF_{(\tilde x, \tilde y)})^{-1} \| \leq \sqrt{|f'(\tilde x)|^2 + 1}$ applied to $(\tilde x, \tilde y) = (\tilde x_3, \tilde y_3)$, yielding
\[
\| (dF_{(\tilde x_3, \tilde y_3)})^{-1}\| \leq \sqrt{L^{2\b/12} + 1} \leq 2 L^{\b/12} \, ,
\]
hence $\| d\tilde F^3 u\| \geq \frac12 L^{\b/3 - \b/12} = \frac12 L^{\b/4} \gg L^{\b/5}$, completing the proof.
\end{proof}

\section{The standard map}

Let $\psi$ and  $f_0 = f_{\psi, L, a} = L \psi + a$ be as defined in Sect. 2.1.

\begin{lem}
There exists $\varepsilon > 0$ and $K_0 > 1$, depending only on $\psi$, for which the following holds: for all $L > 0$ and $f \in \mathcal U_{\varepsilon, L}(f_0)$,
\begin{itemize}
\item[(a)] $\max\{\|f'\|_{C^0}, \|f''\|_{C^0},  \|f'''\|_{C^0}\} \leq K_0 L$,
\item[(b)] The cardinalities of $C'_f$ and $C''_f$ are equal to those of $f_0$ (equivalently
those of $\psi$), 
\item[(c)]  $\min_{\hat x \in C_f'} |f''(\hat x)| \, , \  \min_{\hat z \in C_f''} |f'''(\hat z)| \geq K_0^{-1} L$\ , and 
\item[(d)] 
$\min_{\hat x, \hat x' \in C'_f} d(\hat x,\hat x' )\, , \ \min_{\hat z, \hat z' \in C''_f} 
d(\hat z, \hat z') \geq K_0^{-1}$
\end{itemize}
\end{lem}

The proof is straightforward and is left to the reader.

\begin{proof}[Proof of Theorem \ref{thm:lyap3}]
We claim -- and leave it to the reader to check -- that the proofs in Sections 3--5 (with $C_f', C_f''$ replacing $C_\psi', C_\psi''$)
use only the form of the maps $F= F_{f}$ as defined in Sect. 2.1,
and the four properties above. Thus they prove Theorem \ref{thm:lyap3} as well.
\end{proof}

\begin{proof}[Proof of Corollary \ref{cor:standard}]
Under the (linear) coordinate change $x = \frac{1}{2 \pi} \theta, y = \frac{1}{2 \pi}(\theta - I)$, the standard map conjugates to the map
\[
(x,y) \mapsto (L \sin (2 \pi x) + 2 x - y, x) 
\]
defined on $\T^2$, with both coordinates taken modulo $1$. This map is of the form $F_{f}$, with $f(x) = f_0(x) +2  x$ and $f_0 (x) : =L \sin (2 \pi x)$; here $a = 0$ and $\psi(x) = \sin(2 \pi x)$. 
Let $\varepsilon>0$ be given by Theorem \ref{thm:lyap3} for this choice of $\psi$.
Then $f$ clearly belongs in $\mathcal U_{\varepsilon, L}(f_0)$ for large
enough $L$. 
\end{proof}

\bibliography{ref}
\bibliographystyle{siam}

\end{document}